\documentclass[12pt]{article}
\usepackage{amsmath, amsthm, amsfonts, amssymb}

\def\fullpage {
\addtolength{\topmargin}{-2 cm}
\addtolength{\oddsidemargin}{-0.9cm} \addtolength{\textwidth}{+2 cm}
\addtolength{\textheight}{+4 cm}}
\fullpage
\newtheorem{thm}{Theorem}
\newtheorem{lemma}[thm]{Lemma}

\newtheorem{conj}[thm]{Conjecture}
\newtheorem{problem}[thm]{Question}

\theoremstyle{remark}

\theoremstyle{definition}

\title{Sparse hypergraphs with low independence number}

\begin{document}

\author{
\quad{Jeff Cooper}
\thanks{Department of Mathematics, Statistics, and Computer
Science, University of Illinois at Chicago, Chicago, IL 60607, USA;  email:
jcoope8@uic.edu}
\quad{Dhruv Mubayi}
\thanks{Department of Mathematics, Statistics, and Computer Science, University of Illinois at Chicago, Chicago IL 60607, USA; research supported in part by NSF grant DMS-1300138; email:  mubayi@uic.edu}
}

\maketitle

\begin{abstract}
Let $K^{(3)}_4$ denote the complete $3$-uniform hypergraph on $4$ vertices.
Ajtai, Erd\H{o}s, Koml\'os, and Szemer\'edi (1981)
asked if there is a function $\omega(d)\to \infty$ such that every $3$-uniform, $K^{(3)}_4$-free
hypergraph $H$ with $N$ vertices and average degree $d$ has independence number at least 
$\frac{N}{d^{1/2}}\omega(d)$.
We answer this question by constructing a $3$-uniform, $K^{(3)}_4$-free
hypergraph with independence number at most $2\frac{N}{d^{1/2}}$.
We also provide counterexamples to several related conjectures
and improve the lower bound of some hypergraph Ramsey numbers.
\end{abstract}

\section{Introduction}
A $k$-uniform hypergraph $H$ is a pair $H = (V,E)$, where $V$ is the vertex set
and $E \subset \binom{V}{k}$ is the edge set. We refer to the edge
set of the hypergraph by $H$ and the vertex set by $V(H)$.
The degree of a vertex in $V(H)$ is the number of edges containing
that vertex.
An independent set in a hypergraph is a subset of $V(H)$ which
contains no edge of $H$. The independence number of $H$, denoted $\alpha(H)$,
is the maximum size of an independent set in $H$.
Tur\'an~\cite{turan} showed that $\alpha(G) \geq \frac{N}{d+1}$ for any graph $G$
with $N$ vertices and average degree $d$.
Spencer~\cite{turanspencer} extended Tur\'an's result to hypergraphs
by showing that for all $k \geq 1$ there is a $c_k$
so that every $(k+1)$-uniform hypergraph $H$ with average degree $d$
satisfies $\alpha(H) \geq c_k \frac{N}{d^{1/k}}$.

When $G$ is a graph, Tur\'an's bound can be improved if $G$ is
forbidden from containing a fixed subgraph.
Ajtai, Koml\'os, and Szemer\'edi~\cite{trifreeaks} showed that if $G$ is triangle-free,
then 
\begin{equation}\label{trifree}
\alpha(G) \geq \frac{1}{100}\frac{N}{d}\log d.
\end{equation}
Ajtai, Erd\H{o}s, Koml\'os, and Szemer\'edi~\cite{ktfreeaeks} subsequently 
showed that if $t \geq 4$ and $G$ is $K_t$-free, then $\alpha(G) \geq c_t \frac{N}{d}\log\log d$. 

Let $H$ be a $(k+1)$-uniform hypergraph with $N$ vertices and average degree $d$.
Ajtai, Koml\'os, Pintz, Spencer, and Szemer\'edi~\cite{uncrowdedakpss}
showed that there exists a positive constant $c_k$ such that 
if $H$ contains no $2$, $3$, or $4$ cycles, then
\begin{equation}\label{linear}
\alpha(H) \geq c_k \frac{N}{d^{1/k}}\log^{1/k}d.
\end{equation}
Applications of \eqref{linear} have been found in
number theory~\cite{trifreeaks}, discrete geometry~\cite{heilbronn}, coding theory~\cite{lefmann}, 
and Ramsey theory~\cite{rainbow}.
Ajtai, Erd\H{o}s, Koml\'os and Szemer\'edi asked if, like in the graph case,
\eqref{linear} could also be extended to other families of hypergraphs.

\begin{problem}[Ajtai-Erd\H{o}s-Koml\'os-Szemer\'edi~\cite{ktfreeaeks}]\label{aeksconj}
Is there a function $\omega(d)\to \infty$ such that if a $3$-uniform
hypergraph $H$ contains no $K^{(3)}_4$ (or even $K^{-(3)}_4$), then $\alpha(H) \geq \frac{N}{d^{1/2}}\omega(d)$?
\end{problem}
\noindent
We construct hypergraphs which negatively answer this question, even in the $K_4^{-(3)}$ case.
The construction is presented in Section \ref{secconstruction0}. In Section \ref{secconstruction},
we generalize this construction to $k$-uniform hypergraphs and 
disprove several conjectures related to Question \ref{aeksconj}.
We also discuss an application to hypergraph Ramsey numbers.

\section{$3$-uniform construction}\label{secconstruction0}
In this section, we answer Question \ref{aeksconj} by constructing a 
$K_4^{-(3)}$-free, $3$-uniform hypergraph $H$ with independence number at most $2N/d^{1/2}$.
The hypergraph $H$ is constructed from the complete bipartite
graph $K_{n,n}$ with vertex classes $[n]$ and $[n]$. The vertices of $H$ correspond to
edges in the graph, while the edges of $H$ correspond to $3$-edge paths
which open in the increasing direction:
\begin{align*}
V(H) &= [n]\times[n] \\
E(H) &= \{ \{ab, ac, db \in [n]\times[n]: c > b, d > a \}.
\end{align*}
$H$ is clearly $3$-uniform, contains $N = n^2$ vertices,
and has average degree $d = 3(n-1)^2/4$.
For $v \in V(H)$, consider the link graph $L_v = \{uw: uvw \in E(H)\}$.
The components of $L_v$ are either stars (when $v$ is in the role of $ac$) or a bipartite
graph (when $v$ is in the role of $ab$). $K_4^{-(3)}$, on the other hand,
contains a vertex whose link graph contains a triangle, so $H$ must be $K_4^{-(3)}$-free.

Let $S \subset V(H)$. If $|S| \geq 2n$, then the edges in $K_{n,n}$ corresponding
to the vertices in $S$ contain a cycle on at least four vertices. 
The smallest vertex on this cycle is contained in a $3$-edge path which opens
in the increasing direction, and this path corresponds to an edge in $H$.
Therefore, $\alpha(H) < 2n < 2 N/d^{1/2}$.

\section{Related problems and conjectures}\label{secconstruction}
\subsection{Ramsey numbers for $3$-uniform tight paths}

Our construction also provides the correct order of magnitude for some
new $3$-uniform Ramsey numbers. Let $F$ be a $3$-uniform hypergraph.
Recall that the Ramsey number $r(F,t)$ is the smallest $n$ so that every red-blue coloring of the edges
of $K_n^{(3)}$ contains a red $F$ or a blue complete $K_t^{(3)}$.
Let $P_s$ denote the $3$-uniform hypergraph with vertex set $[s+2]$ and 
edge set $\{\{i, i+1, i+2\}: i \in [s]\}$. $P_s$ is called the $3$-uniform tight path.
Results of Phelps and R\"odl~\cite{phelpsrodl} imply that the Ramsey number of $P_2$ satisfies
\[
r(P_2, t) = \Theta(t^2/\log t).
\]

It is easy to prove that for fixed $s$, we have ex$(n, P_s) = O(n^2)$ and this immediately implies that
\[
r(P_s, t)  = O(t^2).
\]
Indeed, if we have a $P_s$-free $3$-uniform hypergraph on $c_sn$ vertices ($c_s$ large), 
then its average degree is at most $c'_s n$, so it has an independent set of size at least 
$t=c''_s n^{1/2}$.

We now show that the construction in Section \ref{secconstruction0} contains no $P_4$,
which improves the lower bound of $r(P_s,t)$ for $s \geq 4$.
The order of magnitude of $r(P_3,t)$ remains open.
\begin{thm}
Fix $s \ge 4$. Then $r(P_s, t) = \Theta(t^2)$.
\end{thm} 
\begin{proof}
We only need to prove the lower bound, which follows by observing 
that the hypergraph $H$ in Section \ref{secconstruction0} contains no $P_4$. 
Recall that every link graph of $H$ has one
component that is a complete bipartite graph and all of its other components are stars;
further, the pairs of vertices which form edges in the bipartite component appear
in exactly one edge of $H$.
On the other hand, the link graph of each of the degree $3$ vertices in
$P_4$ contains a $3$-edge path and one of the pairs of vertices which form an edge
in this path is contained in two edges of $P_4$.
\end{proof}

\subsection{Generalization to $k$-uniform hypergraphs}
The construction in Section \ref{secconstruction0} starts with a bipartite graph
and builds a hypergraph whose edges correspond to $3$-edge paths in the graph. 
In this section, we generalize this
method by starting with a multipartite hypergraph and building a new hypergraph whose
edges correspond to some fixed hypergraph.
The resulting hypergraphs provide counterexamples to various conjectures concerning
$k$-uniform hypergraphs.

\subsubsection{Chromatic number of $k$-uniform hypergraphs}
A proper coloring of a hypergraph $H$ is a partition of $V(H)$ into independent sets.
The chromatic number of $H$, denoted $\chi(H)$, is the minimum number of parts needed in a 
proper coloring of $H$.
Erd\H{o}s and Lov\'asz~\cite{erdoslovasz} showed that every $(k+1)$-uniform hypergraph
with maximum degree $\Delta$ has $\chi(H) \leq c_k \Delta^{1/k}$.
Strengthening \eqref{linear}, Frieze and the second author~\cite{fmcoloringk} showed that every
$(k+1)$-uniform linear hypergraph with maximum degree $\Delta$ satisfies
$\chi(H) \leq c'_k (\frac{\Delta}{\log\Delta})^{1/k}$.
In \cite{fmcoloring3,fmcoloringk}, 
the same authors conjectured a stronger positive answer to the question of
Ajtai, Erd\H{o}s, Koml\'os, and Szemer\'edi.
\begin{conj}[Frieze-Mubayi~\cite{fmcoloring3, fmcoloringk}]\label{fmconj}
If $F$ is a $(k+1)$-uniform hypergraph 
and $H$ is an $F$-free $(k+1)$-uniform hypergraph with maximum degree $\Delta$, 
then $\chi(H) \leq c_{F}(\Delta/\log \Delta)^{1/k}$. 
\end{conj}

Let $T_k$ be the $k$-uniform hypergraph with $k+1$ edges $e_1, \ldots, e_k,f$ 
where for all $i \neq j$ we have $e_i \cap e_j=S$ and $f \supset e_i-S$ 
for some $S$ with $|S|=k-1$.
In other words, $k$ edges share the same set of $k-1$ points and the last edge contains the remaining 
vertex from each of the $k$ edges.  A $k$-uniform hypergraph has independent neighborhoods if it contains no copy of $T_k$. 
Bohman, Frieze, and the second author~\cite{bfmhfree} conjectured a weaker version of 
Conjecture \ref{fmconj}:
if $H$ is a $3$-uniform hypergraph with maximum degree $\Delta$ and 
independent neighborhoods, then $\chi(H) = o(\Delta^{1/2})$. 

The construction in Section \ref{secconstruction0} shows that both of these
conjectures are false for $3$-uniform hypergraphs. We now generalize that
construction to disprove these conjectures for $k$-uniform hypergraphs.

\subsubsection{Construction from positive strong $k$-simplices}\label{secconstructionsimplex}
Fix $k \ge 2$.  
A \emph{$k$-simplex} is a collection of $k+1$ sets with empty intersection,
every $k$ of which have nonempty intersection. 
A \emph{strong $k$-simplex} $S_k$, introduced in \cite{clusterm}, 
is the $k$-uniform hypergraph with vertex set 
$\{v_1, v_1', .\ldots, v_k, v_k'\}$ and edge set  $\{e, e_1, \ldots, e_k\}$ where  
$e=\{v_1, \ldots, v_k\}$ and $e_i=e \cup\{v_i'\}-v_i$ ($e$ is called the central edge). 
Given disjoint sets $X_1, \ldots, X_k$ with each $X_i\cong [n]$, 
a \emph{positive strong $k$-simplex} $S_k^+$ is a $k$-partite strong simplex 
satisfying $v_i, v_i' \in X_i$ and $v'_i > v_i$ for each $i=1,\dots,k$.

Let $X_1, \ldots, X_k$ be disjoint sets each isomorphic to $[n]$. 
Define the $(k+1)$-uniform hypergraph $H_k$ with vertex set $X_1\times \cdots \times X_k$
and edge set
\[
H_k = \{A \subset X_1\times \cdots \times X_k: A \cong S_k^+ \}.
\]
For example, $H_2$ corresponds to the construction in Section \ref{secconstruction0}.

Fix a $k$-uniform hypergraph $F_k$.
The Zarankiewicz number $z(n, F_k)$ is the maximum number of edges in a $k$-partite $k$-uniform 
hypergraph with parts of size $n$ that 
contains no copy of $F_k$. Since copies of $S_k^+$ correspond to edges of $H_k$,
\[
\alpha(H_k) \leq z(n, S_k^+).
\]
We may thus use the following lemma below to bound $\alpha(H_k)$.

\begin{lemma}\label{positivesimplex}
Fix $k \ge 2$. Then $z(n, S^+_k) \le  2kn^{k-1}$.
\end{lemma}

\begin{proof}
We proceed by induction on $k$. 
The base case $k=2$ follows from Section \ref{secconstruction0}.
For the induction step, suppose we are given a $k$-partite 
$H \subset X_1 \times \dots \times X_k$ with  $|H| > 2kn^{k-1}$, where each $X_i \cong [n]$.
For a vertex $v$ in a $k$-uniform hypergraph $H$, define its link to be the $(k-1)$-uniform hypergraph
$L_v=\{S \subset V(H): v \not\in S, S \cup \{v\} \in H\}$.  
For a set of vertices $T$,  let $d_H(T)$ denote the number of edges containing $T$.

For each $v \in X_1$, let $L_v$ be the link $(k-1)$-uniform hypergraph of $v$.  
Let $A_v \subset L_v$ comprise those $(k-1)$-sets $T$ with $d_H(T)=1$ and $B_v=L_v-A_v$.

Let $B_v^+$ be the set of all $S \in B_v$ 
such that there exists $v'>v$ with $S \in L_{v'}$.  
We will find $x \in X_1$ with $|B^+_x| > 2(k-1)n^{k-2}$ and then apply induction. 
Now
\[
\sum_{v \in X_1} |B^+_v|
= \sum_{\substack{S \in X_2\times\cdots\times X_k:\\d_H(S)\ge 2}} (d_H(S) -1)
\ge \sum_{\substack{S \in X_2\times\cdots\times X_k:\\d_H(S)\ge 2}} d_H(S)-n^{k-1}
= \sum_{v \in X_1} |B_v|-n^{k-1}.
\]
Thus
\begin{align*}
2kn^{k-1} 
< |H|
=\sum_{v \in X_1}|L_v|
= \sum_{v \in X_1}|A_v|+\sum_{v \in X_1}|B_v|
&\leq n^{k-1}+\sum_{v \in X_1}|B_v| \\
&\leq 2n^{k-1}+\sum_{v \in X_1}|B^+_v|.
\end{align*}

Consequently, there exists $x \in X_1$ with  $|B^+_x| > 2(k-1)n^{k-2}$. 
Apply induction to $B_x$ to obtain a copy of $S^+_{k-1}$ in $X_2 \times \cdots \times X_{k}$. 
To form $S^+_k$, begin by enlarging each edge of $S^+_{k-1}$ with $x$. 
Add another edge by enlarging the central edge $e$ by some other vertex $y \in X_1$
with $y > x$. Note that $y$ exists since $e \in B_x^+$ and $d_H(e)>1$.  
We have thus obtained a copy of $S_k^+$, where $e \cup \{x\}$ is the central edge.
\end{proof}

Notice $H_k$ has $N = n^k$ vertices and maximum degree $\Delta \leq (k+1)n^k$.
By Lemma \ref{positivesimplex},
\[
\alpha(H_k) 
\leq 2kn^{k-1}
= 2k (k+1)^{1/k} \frac{n^k}{((k+1)n^k)^{1/k}}
\leq 2k (k+1)^{1/k} \frac{N}{\Delta^{1/k}},
\]
and
\[
\chi(H_k) \geq \frac{\Delta^{1/k}}{2k(k+1)^{1/k}}.
\]

Recall that $T_{k+1}$ is the $(k+1)$-uniform hypergraph with $k+2$ edges $e_1, \ldots, e_{k+1},f$ 
where for all $i \neq j$, $e_i \cap e_j=S$ and $f \supset e_i-S$ for some $S$ with $|S|=k$.
Suppose $S_{k,1}^+, \dots, S_{k,k+1}^+$ satisfy $S_{k,i}^+ \cap S_{k,j}^+ = S$, for $i \neq j$ and $|S| = k$.
Since each strong $k$-simplex is positive, they must share a single central edge.
Thus the edges in $(S_{k,1}^+ \cup \dots \cup S_{k,k+1}^+)-S$ share a single vertex
and so do not form a positive strong $k$-simplex.
Therefore $H_k$ does not contain any copy of $T_{k+1}$, 
disproving Conjecture \ref{fmconj} and the weaker conjecture of \cite{bfmhfree}.

\subsubsection{$c$-sparse hypergraphs}
A hypergraph is \emph{$c$-sparse} if every vertex subset $S$ spans at most $c|S|^2$ edges.
By Spencer's extension of Tur\'an's bound, every $c$-sparse hypergraph $H$ with $N$ vertices satisfies
$\alpha(H) \geq c'_k \sqrt{N}$.
Phelps and R\"odl~\cite{phelpsrodl} improved this to
$\alpha(H) \geq c'_k \sqrt{N \log N}$ for linear $3$-uniform hypergraphs.
In 1986, de Caen (see \cite{decaen}) conjectured that a similar improvement
holds even for $c$-sparse hypergraphs (observe that linear implies $\frac{1}{2}$-sparse). 
\begin{conj}[De Caen~\cite{decaen}]\label{decaenconj}
For every 
positive $c$, there is a function $\omega(N) \rightarrow \infty$ such that every $c$-sparse 3-uniform 
hypergraph $H$ with $N$ vertices satisfies $\alpha(H) \geq \omega(N)\sqrt{N}.$
\end{conj}

\noindent
Recently, Kostochka, the second author, and Verstra\"ete~\cite{kmvdecaen} 
posed a stronger version of de Caen's conjecture:
for every positive $c$, there is a function $\omega(N) \rightarrow \infty$ such that
every $c$-sparse $3$-uniform hypergraph $H$ with $N$ vertices and average degree $d$
satisfies $\alpha(H) \geq \omega(N) \frac{N}{d^{1/2}}$.

Observe that the construction in Section \ref{secconstruction0} is $1$-sparse, 
so $S_2$ immediately provides a counterexample to Conjecture \ref{decaenconj} and
the conjecture of \cite{kmvdecaen}.
However, for $k \geq 3$, $H_k$ is not $c$-sparse for any constant $c$,
so one may ask whether or not for $k \geq 3$ and every positive $c$ there is a function
$\omega(N) \rightarrow \infty$ such that every $c$-sparse $(k+1)$-uniform 
hypergraph $H$ with $N$ vertices satisfies $\alpha(H) \geq \omega(N)N^{1/k}.$
The next section provides a counterexample to this generalization of de Caen's conjecture.

\subsubsection{Construction from special $k$-clusters}\label{hdk}
A \emph{$k$-cluster}, introduced in \cite{clusterm}, is a collection of $k+1$ sets with 
empty intersection whose union has size at most $2k$.
The family of \emph{special $k$-clusters} $\mathcal{D}_k$ is the $k$-uniform hypergraph 
family that is defined inductively as follows:
$\mathcal{D}_2 = \{D_2\}$, where  $D_2$ is the path with three edges.
For $k \geq 3$, $\mathcal{D}_k$ is the family of $k$-uniform hypergraphs which
can be constructed as follows: begin with any $D_{k-1} \in \mathcal{D}_{k-1}$, 
which is assumed inductively to have $2(k-1)$ vertices and two disjoint edges $a$ and $b$.
Then $D_k$ is a member of $\mathcal{D}_k$ if it can be formed 
by adding two new vertices $x,y$ to $D_{k-1}$, enlarging all edges of $D_{k-1}$ 
by including $x$, and enlarging $a$ by including $y$. 
Thus $D_k$ has $2k$ vertices and $k+1$ edges, two of which are disjoint. 
We will use $D_k$ to denote an arbitrarily chosen member of $\mathcal{D}_k$.

Following the construction from Section \ref{secconstructionsimplex},
define the $(k+1)$-uniform hypergraph $J_k$ with vertex set $X_1\times \cdots \times X_k$
and edge set
\[
J_k = \{A \subset X_1\times \cdots \times X_k: A \cong D_k \}.
\]

\begin{lemma}\label{kcluster}
Fix $k \ge 2$ and $D_k \in \mathcal{D}_k$. Then $z(n, D_k) \le  kn^{k-1}$.
\end{lemma}

\begin{proof}
We proceed by induction on $k$.
For the base case $k=2$, observe that $D_2$ is the path with 3 edges. 
If $H$ is a bipartite graph with more than
$2n$ edges, then $H$ contains a cycle with at least four edges, which contains a copy of $D_2$.
For the induction step, suppose we are given $k$-partite $H$ with  $|H|>kn^{k-1}$ with 
parts $X_1, \ldots, X_k$ each of size $n$. 
For each $v \in X_1$, define $L_v$, $A_v$, and $B_v$ as in the proof of Lemma \ref{positivesimplex}.
Then
\begin{align}\label{bvbound}
kn^{k-1} < |H| 
=\sum_{v \in X_1}|L_v|
= \sum_{v \in X_1}|A_v|+\sum_{v \in X_1}|B_v|
&\le n^{k-1}+\sum_{v \in X_1}|B_v|.
\end{align}
Consequently, there exists $x \in X_1$ with  $|B_x|>(k-1)n^{k-2}$. Let $D_{k-1}$ be the member of 
$\mathcal{D}_{k-1}$ that gives rise to $D_k$ in the inductive construction of $D_k$. 
Apply induction to $B_x$ to obtain  a copy of $D_{k-1}$ in $B_x$.
To form $D_k$, begin by  enlarging each edge of $D_{k-1}$ with $x$. 
Add another edge by enlarging one of the two disjoint edges $a,b$ of $D_{k-1}$ (say $a$) 
by some other vertex $y \in X_1$. Note that $y$ exists since $a \in B_x$.
We have thus obtained a copy of $ D_k$, where $a\cup \{y\}$ and $b\cup \{x\}$ are the disjoint edges. 
\end{proof}

By Lemma \ref{kcluster}, $\alpha(J_k) \leq kn^{k-1}$,
so it suffices to show that $J_k$ is $2^{2k^2-2k-1}$-sparse.
Let $S \subset V(J_k)$.
The vertex set of a copy of $D_k$ is determined by the
two disjoint edges in $D_k$. There are at most $\binom{2k}{k}^{k-1}$
possibilities for the remaining $k-1$ edges.
Therefore we may associate every pair of vertices in $S$ 
to at most $\binom{2k}{k}^{k-1}$ edges in the subgraph induced by $S$.
Since every edge corresponds to at least one pair of vertices,
the number of edges in $S$ is at most
\[
\binom{2k}{k}^{k-1}\binom{|S|}{2} < 2^{2k^2 - 2k - 1}|S|^2.
\]
This disproves the generalization of de Caen's conjecture to $k$-uniform hypergraphs.

\section{Concluding remarks}
\begin{itemize}
\item 
$H_k$ is a counterexample to Conjecture \ref{fmconj} with $N$
vertices and maximum degree $\Theta(N)$.
Sparser counterexamples with $f N$ vertices and maximum degree $N$ 
can be constructed by taking the disjoint union of $f$ copies of $H_k$.

\item 
Benny Sudakov suggested the following generalization of $H_2$ to $(k+1)$-uniform
hypergraphs, which provides a denser counterexample to Conjecture \ref{fmconj} for $k \geq 3$. 
Let $G$ be the $(k+1)$-uniform hypergraph with vertex set $[n]\times[n]$ and
edge set
\[
\{(x_1, y_1), (x_1, y_2), (x_2, y_2), \dots, (x_k, y_2): 
 y_2 < y_1, x_i < x_{i+1} \text{ for } i \in [k-1] \}.
\]
In other words, each edge corresponds to an $L$ with $k$ points on its base.
It is not hard to see that $G$ has maximum degree $\Theta(n^{k})$,
independence number $\Theta(n)$, and contains no copy of $T_{k+1}$.

\item
For $1<r<k+1$, say that a $(k+1)$-uniform hypergraph is $(c,r)$-sparse if every vertex subset 
$S$ spans at most $c|S|^r$ edges.  
A partial Steiner $(k+1,k)$-system is a $(k+1)$-uniform hypergraph with every $k$ vertices 
in at most one edge. 
Such a system has average degree at most $n^{k-1}$ and, by \cite{kmvdecaen}, has independence number at 
least $c'(n \log n)^{1/k}$ for some positive $c'$. 
This result cannot be extended to the larger class of $(c,k)$-sparse $(k+1)$-uniform hypergraphs,
as shown by the following $(c,3)$-sparse 4-uniform hypergraphs with independence number $O(n^{1/3})$.

Let $F$ be the set of 3-partite 3-uniform hypergraphs with four edges such that one of the edges is 
contained in the union of the other three.  
Then it is an easy exercise to show (by induction on $n$ for example) that $z(n, F)= O(n)$,
so our general construction provides 
a 4-uniform, $(c,3)$-sparse  hypergraph $H(F)$ on $n^3$ vertices 
with $\alpha(H(F)) = O(n)$ (for $(c,3)$-sparse, use the argument in Section \ref{hdk}).

We remark that, in addition, $H(F)$ contains no $K_{163}^{(4)}$ for the vertex set of a
copy of $K_{163}^{(4)}$ would correspond to a set of  $163 = 1+3!(4-1)^3$ 3-uniform edges, 
and by the Erd\H{o}s-Rado sunflower lemma, these edges would contain a sunflower $C$ of size 4. 
But the 4 vertices in $H(F)$ corresponding to the edges of $C$  cannot form an edge 
in $H(F)$ since not one of them is contained in the union of the other three.  

\item
Define the $3$-uniform hypergraphs 
$F_5 = \{ abc, abd, cde \}$ and $C_3 = \{abc, cde, efa\}$.
The authors~\cite{cmcoloring3} recently answered Question \ref{aeksconj} positively
if, in addition to $K^{-(3)}_4$, $F_5$ and $C_3$ are also forbidden.
It would be interesting to answer Question \ref{aeksconj} if only $K^{-(3)}_4$ and $C_3$ are forbidden.

\end{itemize}

\section{Acknowledgments}
We would like to thank the referees for carefully reading our manuscript and 
providing thoughtful feedback.

\bibliographystyle{amsplain1}
\bibliography{bib}

\end{document}